\documentclass[11pt]{amsart}

\usepackage{amscd,amsfonts,amsmath,amsthm,anyfontsize,bigints,blindtext,centernot,cite,colortbl,enumitem,extarrows,fancyhdr,filecontents,float,footmisc,galois,geometry,hyperref,imakeidx,lastpage,lipsum,lmodern,marginnote,mathrsfs,mathtools,MnSymbol,pgfplots,spectralsequences,stmaryrd,subcaption,thmtools,thm-restate,tikz,tikz-cd,titlecaps,titlesec,xcolor}
\usepackage[utf8]{inputenc}
\usepackage[T1]{fontenc}
\makeindex[columns=1,title=Index] % index

\hypersetup{
    colorlinks,
    linkcolor={red!50!black},
    citecolor={blue!70!black},
    urlcolor={green!80!black}
}
\usetikzlibrary{positioning} % part of tikz?

\newlist{subitems}{itemize}{1}
\setlist[subitems]{label={--},nosep}
\newlist{myreferences}{enumerate}{1}
\setlist[myreferences]{label={[\arabic*]},leftmargin=*}

\newcommand{\id}{\text{id}} \newcommand{\interior}{\text{int\;}}

 \newcommand{\ima}{\text{im\,}}

\newcommand{\C}{\mathbb{C}} \newcommand{\R}{\mathbb{R}}
 \newcommand{\Z}{\mathbb{Z}}
\newcommand{\N}{\mathbb{N}} 
\newcommand{\T}{\mathbb{T}}
 \newcommand{\sph}{\mathbb{S}} \newcommand{\D}{\mathbb{D}}
 \newcommand{\calO}{\mathcal{O}}
 \newcommand{\cp}{\C P}
\newcommand{\F}{\mathscr{F}} 

\newtheorem{theorem}{Theorem}[section]
\newtheorem{proposition}[theorem]{Proposition}
\newtheorem{definition}[theorem]{Definition}
\newtheorem{corollary}[theorem]{Corollary}
\newtheorem{lemma}[theorem]{Lemma}

\newtheorem{example}[theorem]{Example}
\newtheorem{property}[theorem]{Property}

\theoremstyle{remark}
\newtheorem{remark}[theorem]{Remark}

\titleformat{\section}{\normalfont\normalsize\bfseries}{\thesection}{1em}{}
\titleformat{\subsection}{\normalfont\normalsize}{\thesubsection}{1em}{}
\titleformat{\subsubsection}{\normalfont\normalsize}{\thesubsubsection}{1em}{}
\begin{document}
\noindent\begin{tabular}{@{}p{\linewidth}@{}}
 \centering\LARGE Structural Stability for Fibrewise Anosov Diffeomorphisms on Principal Torus Bundles\\[5pt]
 \centering\large Danyu Zhang\\
% \centering\large \\
% \large\today\\
\end{tabular}

\thispagestyle{empty}
%\noindent\tableofcontents

\begin{abstract} We show a fibre-preserving self-diffeomorphism which has hyperbolic splittings along the fibres on a compact principal torus bundle is topologically conjugate to a map that is linear in the fibres.
\end{abstract}
\makeatletter
\@setabstract
\makeatother

\

\section{Introduction}\label{intro}

Let $M$ be a closed Riemannian manifold. Recall that a diffeomorphism $f:M\to M$ is called \emph{Anosov} if it satisfies the following conditions.
\begin{itemize}
    \item[1.] There is a splitting of the tangent bundle $TM=E^s\oplus E^u$ which is preserved by the derivative $df$.
    \item[2.] There exist constants $C>0$ and $\lambda\in(0,1)$, such that for all $n>0$, we have 
    $$\|df^nv\|\leq C\lambda^n\|v\|\ \ \text{ for all } v\in E^s$$
    and
    $$\|df^{-n}v\|\leq C\lambda^n\|v\|\ \ \text{ for all } v\in E^u.$$
    \end{itemize}

Classification of Anosov diffeomorphisms is a well-known open problem. Examples are only known on tori, nilmanifolds and infranilmanifolds. The very first examples of Anosov diffeomorphisms are hyperbolic automorphisms which are given by hyperbolic matrices in $GL(d,\Z)$ whose action on $\R^d$ descends to the torus $\T^d$. J. Franks \cite{franks1} proved that every Anosov diffeomorphism which is homotopic to a hyperbolic automorphism is, in fact, conjugate to this automorphism. A. Manning \cite{manning1} then completed the classification on tori (and more generally on infranilmanifolds) by proving that every Anosov diffeomorphism on the torus is homotopic to a hyperbolic automorphism. We would usually refer to such a property as the \emph{global structural stability}. More precisely, there is a linear model for every Anosov diffeomorphism. The word ``global'' is in contrast to the local version of structural stability, where we consider whether two diffeomorphisms that are close in $C^r$-topology for some $r\geq 0$ are conjugate.

Now suppose $p:E\to B$ is a $C^1$ principal torus bundle where $E$ and $B$ are smooth compact manifolds. Roughly speaking, $E$ is a topological space endowed with a free action of the torus $\T^d=\R^d/\Z^d$. We will include a more detailed definition and properties of this object in the preliminary section below. Fix a Riemannian metric $\|\cdot\|$ on $E$. Then the tangent space splits $TE=V\oplus H$, where $V=\ker(dp)$ denotes the vertical bundle tangent to the fibres and $H$ the horizontal bundle which is its orthogonal complement. We call a $C^1$ diffeomorphism $F:E\to E$ a \emph{fibrewise Anosov diffeomorphism} if it satisfies the following conditions.
\begin{itemize}
    \item[1.] There exists a splitting of the vertical bundle $V=V^s\oplus V^u$ which is preserved under the derivative $dF$. (In particular $V$ is $dF$-invariant.)
    \item[2.] There exist constants $C$ and $\lambda\in (0,1)$, such that for all $n>0$ we have
\begin{align*} \|dF^n(v^s)\|\leq C\lambda^n\|v^s\|,\ \ \ \text{for all}\ v\in V^s,
\end{align*}
and
\begin{align*} \|dF^{-n}(v^u)\|\leq C\lambda^n\|v^u\|\ \ \ \text{for all}\ v\in V^u.
\end{align*}
\end{itemize}

If $G:E\to E$ is a fibre-preserving map on a principal $\T^d$ bundle that satisfies $G(e.g)=G(e).Ag$, for some $A\in GL(d,\Z)$ and all $e\in E,g\in\T^d$ in the structure group, then we call $G$ an \emph{$A$-map}. In particular, when $A$ is hyperbolic, we call $G$ a \emph{fibrewise affine Anosov diffeomorphism}.

We took this notion of fibrewise hyperbolicity from \cite{f-g} where Farrell and Gogolev mainly studied the bundles that support such dynamics. There is also a more general notion of ``foliated hyperbolicity'' by Bonatti, Gómez-Mont and Martínez that appeared in \cite{bgmm} where the authors studied diffeomorphisms that are hyperbolic along the leaves of a foliated manifold and proved elementary dynamical properties, e.g., the strong stable and unstable distributions are integrable.

We are interested in this class of fibrewise Anosov examples also partially because of the new class of partially hyperbolic systems constructed by Gogolev, Ontaneda and  Rodriguez-Hertz in \cite{GORH} with simply-connected total spaces, although note that here a fibrewise Anosov diffeomorphism is not necessarily partially hyperbolic unless the base dynamics are ``dominated'' by the dynamics in the fibres.

In this paper, we prove the global structural stability for fibrewise Anosov diffeomorphisms, generalizing the results of Franks and Manning. Namely, we prove the following theorem.

\begin{theorem}\label{main} Let $p:E\to B$ be a compact $C^1$ principal torus bundle and let $F:E\to E$ be a fibrewise Anosov diffeomorphism. Then there exists a fibrewise affine Anosov diffeomorphism $G:E\to E$ and a homeomorphism $h:E\to E$, which is homotopic to $\id_E$ and fibres over $\id_B$, such that $h\circ F=G\circ h$.
\end{theorem}

\begin{remark} Note that because the structure group of a compact principal torus bundle only contains translations, a fibre-preserving map induces the same automorphism on the homology of the torus fibre over each point, provided that the base is connected. This is because the restrictions of the map to nearby fibres are homotopic, and then we can conclude by going from neighborhood to neighborhood.
\end{remark}

We will get the semiconjugacy from the following proposition. 

\begin{proposition}\label{semiii} Suppose $F$ is a continuous fibre-preserving map on a compact principal torus bundle whose base is connected, which induces an automorphism $A:H_1(\T^d)\to H_1(\T^d)$ on the first homology of the torus fibre. Then it is fibrewise homotopic to an $A$-map.

If in addition we assume the bundle is $C^1$, and $A$ is hyperbolic, then $F$ is semi-conjugate to a fibrewise affine Anosov diffeomorphism.
\end{proposition}

We give some definitions in Section \ref{pre}, and examples of fibrewise Anosov diffeomorphisms in Section \ref{examples}. We will first prove in Proposition \ref{mprop} that every fibrewise Anosov diffeomorphism is homotopic to a fibrewise affine Anosov diffeomorphism by topological arguments in Section \ref{homotopy}. Then we prove in Proposition \ref{semii} that there is a semiconjugacy from each map covering the same map in the homotopy class of a fibrewise affine Anosov diffeomorphism, by applying a similar argument as in Franks' proof. Last, we establish the global structural stability at the end of Section \ref{pofmain}, by showing that the fibrewise lift of each pair of stable and unstable leaves has a unique intersection.

\section{Preliminaries}\label{pre}

In this section, we give definitions and properties that will come into use.

\begin{definition}[(\cite{kobanomi})] Let $G$ be a Lie group and $E$ a metric space where $G$ acts continuously, freely and properly on the right. Let $B=E/G$. Then we call the projection $p:E\to B$ a principal fibre bundle over $B$ with group $G$, and $E$ the total space, $B$ the base space, $p^{-1}b$ the fibre for each $b\in B$, and $G$ the structure group.
\end{definition}

In Theorem \ref{main}, by a $C^1$ principal torus bundle, we mean that $p:E\to B$ is a $C^1$ map.

\begin{property}[(\cite{palais})]\label{psi} The principal bundle $E$ defined as above is locally trivial, that is, every point $x\in B$ has a neighborhood $U$ where there exists a homeomorphism $\varphi: U\times G\to p^{-1}(U)$ and a map $\psi:p^{-1}(U)\to G$ so $\varphi^{-1}(e)=(p(e),\psi(e))$ and $\psi(e. g)=\psi(e). g$, where $e\in E$ and $g\in G$.

From now on we assume $E$ is a smooth manifold and the chart $\varphi$ is smooth.
\end{property}

\begin{property} The fibres of a principal bundle with group $G$ are diffeomorphic to $G$.
\end{property}

\begin{property}\label{cocycle} We fix trivializations $\{(\varphi_i,U_i)\}$ for the fibre bundle $p:E\to B$ with structure group $G$ and fibre $F$. Restricting $\varphi_i$ and $\varphi_j$ to $U_i\cap U_j$, there exists a unique map $g_{ji}:U_i\cap U_j\to G$ such that $\varphi_j^{-1}\varphi_i(b,x)=(b,g_{ji}(b)x)$ for $(b,x)\in (U_i\cap U_j)\times F$. The functions $g_{ji}$ satisfy the following properties:
\begin{itemize}\item[1.] For each $b \in U_i\cap U_j\cap U_k$ we have $g_{ki}(b)=g_{kj}(b)g_{ji}(b)$.
\item[2.] For each $b\in U_i$, $g_{ii}(b)=\id_G$.
\item[3.] For each $b\in U_i\cap U_j$, $g_{ij}(b)=g_{ji}^{-1}(b)$.
\end{itemize}
\end{property}

From now on, we will use additive notation ``$+$'' for the torus action.

\begin{remark} Recall the definition of a fibrewise Anosov diffeomorphism $F:E\to E$ on a principal torus bundle from Section \ref{intro}. The assumption that $F$ preserves the vertical bundle implies that $F$ is also fibre preserving. This can be seen as follows.

Let $e_0\in E$, $p(e_0)=b$, and suppose $\psi(e_0)=0$ with the notion in Property \ref{psi}. The fibre $\T^d$ is a connected compact Lie group where the exponential map is globally surjective. We identify the Lie algebra of left invariant vector fields on $\T^d$ and the tangent space $T_0\T^d$.  For any $e_1\in p^{-1}b$, we have an integrable curve $\gamma(t)=\exp(tX)$ in $p^{-1}b$, where $X\in T_0\T^d$, such that $\gamma(0)=e_0$ and $\gamma(1)=e_1$. Since $dF(\frac{d}{dt}|_{t=t_0}(e_0+\exp(tX))=\frac{d}{dt}|_{t=t_0}F(e_0+\exp(tX))\in V_{F(e_0+\exp(t_0X))}$ for all $t_0\in\R$, and $V$ is integrable whose intgrable manifolds at a point $e$ is just the fibre at $pF(e)$, we have $F(e_0+\exp(X))=F(e_1)\in p^{-1}(p(F(e_0)))$.
\end{remark}

We denote the map covered by $F$ as $f:B\to B$.

We would also like to recall the following lemma from general topology which we will refer to as the Tube lemma.

\begin{lemma}[(The Tube Lemma)] Consider the product space $X\times Y$, where $Y$ is compact. If $N$ is an open set of $X\times Y$ containing the slice $x_0\times Y$ of $X\times Y$, then $N$ contains some tube $W\times Y$ about $x_0\times Y$ about $x_0\times Y$, where $W$ is a neighborhood of $x_0$ in X.
\end{lemma}

%\begin{theorem}[(Pugh Closing Lemma, \cite{pugh})]\label{pclosing} Let $f$ be a diffeomorphism of a compact manifold $M$ and let $x\in M$ be a non-wandering point of $f$. Then any $C^1$-neighborhood $U$ of $f$ contains a diffeomorphism $g\in U$ such that $x$ is a periodic point of $g$.
% \end{theorem}

\begin{remark} The stable and unstable distributions of a fibrewise Anosov diffeomorphism along the fibres are integrable following Theorem 2.6, \cite{bgmm}.
\end{remark}

\begin{remark}[(The local product structure)]\label{LPS} Consider a pair of transverse foliations. If the foliated manifold is compact, there exists a constant $\varepsilon$, such that any open neighborhood of diameter $<\varepsilon$ can be thought of parametrized by the plaques of the leaves.

The fibrewise leaves of a fibrewise Anosov diffeomorphism also have the local product structure. More precisely, for $b\in B$ and any $e,e'\in \T^d_b$ (or $\tilde{e},\tilde e'\in \R^d_b$) which denotes the fibre over $b$, there exists a constant $\varepsilon>0$ such that if $d(e,e')<\varepsilon$ (or $d(\tilde e,\tilde e')<\varepsilon$), then $W^s_b(e)\cap W^u_b(e')$ and $W^u_b(e)\cap W^s_b(e')$ (or $\widetilde W^s_b(\tilde e)\cap\widetilde W^u_b(\tilde e')$ and $\widetilde W^u_b(\tilde e)\cap \widetilde W^s_b(\tilde e')$ ) each contains exactly one point.

Now there is also such a constant even if we lift the leaves to the convering space. Since the bundle is compact and a sufficiently small neighborhood of $\T^d_b$ for each $b$ is evenly covered, by definition of the covering map $\R^d_b\to\T^d_b$, there is a uniform $\varepsilon$ for the bundle with which the local product structure is satisfied. We would refer to $\varepsilon$ ``the constant of the local product structure'' just to save a letter.
\end{remark}

\section{Examples of Fibrewise Anosov Diffeomorphisms}\label{examples}

\begin{example}[(The trivial example)]\label{extriv} Let $E=B\times\T^d$. Then $G:E\to E$, $G(b,t)=(f(b),A(t)+v(b))$ where $f:B\to B$ is a diffeomorphism and $v:B\to\T^d$ is a differentiable function, is a fibrewise affine Anosov diffeomorphism.
\end{example}

\begin{example}[(Nilmanifold automorphisms)]\label{exnil} Let $E=N/\Gamma$ be a nilmanifold where $N$ is a simply connected nilpotent Lie group and $\Gamma\subset N$ a discrete subgroup that acts cocompactly on $N$. Let $Z(N)$ denote the center of $N$. Then $Z(N)/(Z(N)\cap\Gamma)$ is a compact abelian Lie group and thus a torus group. It acts continuously, freely and properly on $E$. Thus $E$ is a principal bundle with torus fibres.

Suppose $A:N\to N$ is an automorphism and $A(\Gamma)=\Gamma$. If the restriction of $A$ to $Z(N)$ is hyperbolic, it induces a hyperbolic toral automorphism on $Z(N)/(Z(N)\cap\Gamma)$, and thus a fibrewise affine Anosov diffeomorphism on $E$.
\end{example}

\begin{remark} There are criteria given in \cite{GORH} whether a bundle would support fibrewise Anosov diffeomorphisms. We want to point out that one is able to construct many examples, with some algebraic topology.
\end{remark}

\begin{example}[($K3$ surface, \cite{GORH})]\label{exk3} Let the base be a $K3$ surface by the Kummer's construction \cite[3.3]{scorpan}. It is given by $X:=\T^2_{\C}\# 16\overline{\cp}^2/\iota$, where $\overline{\cp}^2$ is the complex projective plane with reversed orientation and $\iota$ is the involution induced by the involution of $\T^2_{\C}$, which gives rise to $16$ singularities where we attach the $\overline{\cp}^2$'s.

Identify $\T^2_{\C}=\T^4$, the usual real torus. For any given hyperbolic matrix $A\in SL(2,\Z)$, we are able to take a perturbation $A'$ of $A$, such that $A'\oplus A':\T^2_{\C}\to\T^2_{\C}$ descends to a diffeomorphism $f$ of the quotient space $X$, and there exists a principal $\T^2$ bundle with simply connected total space that admits a fibrewise affine Anosov diffeomorphism with matrix $A^2$ as specified in the definition, which covers $f:X\to X$. In particular, the above fibrewise affine Anosov diffeomrphism is partially hyperbolic by this construction.
\end{example}

We would also like to remark that from \cite{surgery}, we are able to construct more examples from the above examples \ref{extriv}, \ref{exnil} and \ref{exk3} by connect-summing along invariant tori.

More precisely, in these examples, the base dynamics (up to a finite order) already have or can be perturbed to have a hyperbolic fixed point, near which the map has a local form $(x,y)\mapsto (Ax,f_x(y))$ in a tubular neighborhood $\D\times\T^d$. Then by replacing $\D\times\T^d$ of the invariant fibre by $\tilde{\D}\times \T^d$ where $\tilde{\D}:=\{(x,l(x)):x\in\D,\ x\in l(x),\ l(x)\text{ is the line passing through } x\}$, we obtain what we call a blow-up of the original bundle, with boundary $\sph^k\times\T^d$ for some $k$.

Now if we have examples of different types, with the same fibre and base dimensions, whose base dynamics have hyperbolic fixed points where the local forms of the maps as described above are the same, then we can take the blow-ups and glue along the boundaries to get new examples of fibrewise affine Anosov diffeomorphisms. We get partially hyperbolic diffeomorphisms from gluing partially hyperbolic diffeomorphisms by \cite{surgery}.

\section{The Homotopy}\label{homotopy}

In this section, we let $F:E\to E$ be a $C^1$ fibrewise Anosov diffeomorphism on the principal $\T^d$ bundle $p:E\to B$ that covers $f:B\to B$. We prove the following proposition.

\begin{proposition}\label{mprop} There is a fibrewise affine Anosov diffeomorphism $G:E\to E$ such that $F$ is homotopic to $G$.
\end{proposition}

We will prove the above proposition by the end of this section. First we want to pick a recurrent point on the base, and show the projection of the map to the fibre over the recurrent point is Anosov. We start with the usual cone field argument \cite[Chapter 6]{k-h}.

In the following, if we take a point $e\in E$ and let $b=p(e)=:b_0$ denote the projection, we will always denote $b_n=f^n(b)$ for $n\in\Z$.

We use $V_e$ to denote the vertical tangent space at the point $e\in E$. Consider the restriction of the $n$-th iteration of $F$ to a fibre $F_{b_0}^n:p^{-1}b_0\to p^{-1} b_n$. Denote by $dF^n_{b_0}:V_{b_0}\to V_{b_n}$ the restriction of the differential of $F^n_{b_0}$ to the vertical bundle tangent to $p^{-1}b_0$ where $V_b=\cup_{e\in p^{-1}b}V_e$.

Let $p^{-1}U$ be a trivialized neighborhood. Let $\varphi:p^{-1}U\to U\times\T^d$ denote the coordinate chart, and $\psi: p^{-1}U\to\T^{d}$ be as in Property \ref{psi}.

Suppose $e\in p^{-1} U$ and $b\in U$. Since we have assumed that there is a splitting $V_b=E_b^s\oplus E_b^u$, for any $v\in V_b$ we can write $v=v^s+v^u$ where $v^s\in E^s_b$ and $v^u\in E^u_b$.
By assumption we have a metric on $TE$ such that for all $v^s\in E^s$, $v^u\in E^u$ and all $n\geq 0$, $\|dF^nv^s\|\leq C\lambda^n\|v^s\|$ and $\|dF^{-n}v^u\|\leq C\lambda^{n}\|v^u\|$ with some $C>0,\lambda\in (0,1)$.

\begin{remark} We could always assume $C=1$ in the following sense. Let $\lambda<\lambda'<1$. We can take the adapted metric such that for two vectors $u,v\in V^s$,
$$\langle u,v\rangle'=\sum_{k=0}^{\infty}(\lambda')^{-2k}\langle dF^k u,dF^kv\rangle$$
for two vectors $u,v\in V^u$,
$$\langle u,v\rangle'=\sum_{k=0}^{\infty}(\lambda')^{2k}\langle dF^{-k}u,dF^{-k}v\rangle,$$
and we set for $u\in V^s,v\in V^u, w\in H$,
$$\langle u,v\rangle'=\langle u,w\rangle'=\langle v,w\rangle'=0.$$
Then we can check the metric is a well-defined metric (the series is bounded) and for instance if $v\in V^s$
\begin{align*} \|dF^nv\|'^2 &= \sum_{k=0}^\infty (\lambda')^{-2k}\|dF^{n+k}v\|^2\\
&= \sum_{k=0}^\infty (\lambda')^{-2k-2n}\|dF^{k}v\|^2-\sum_{k=0}^{n-1}(\lambda')^{-2k}\|dF^kv\|^2\\
&\leq (\lambda')^{-2n}\sum_{k=0}^\infty (\lambda')^{-2k}\|dF^kv\|^2=(\lambda')^{-2n}\|v\|'^2;
\end{align*}
similarly for any $v\in V^u$.
\end{remark}

Now assume the dimensions of $V$, $E^s$ and $E^u$ are $d$, $l$ and $d-l$, respectively. Take an arbitrary local basis of $E^s$ by $p_1,\ldots,p_l$ and of $E^u$ by $p_{l+1},\ldots,p_d$. For two points $e,e'\in p^{-1}U$, suppose $\psi(e)=\psi(e')=x\in\T^d$. Use $d\varphi\circ d\psi:p^{-1}U\to V$ to identify the vertical tangent spaces at $e$ and $e'$. If $p_{i,e}$ is a basis element for $V_e$, then $d\varphi\circ d\psi (p_{i,e})$ is a corresponding basis element for $V_{e'}$, for $i=1,\ldots,d$. But note that this might be different from $p_{i,e'}$, which we picked from a basis of the the stable or unstable subbundle, and they vary continuously as $e$ varies. We can identify the vetical tangent spaces with $\R^d$.

Let $v=v^s+v^u\in E^s\oplus E^u$ where $v^s\in E^s$ and $v^u\in E^u$. Define $C_{u,e}^\gamma$, $C_{s,e}^\gamma$, the cones at a point $e\in E$ with $0<\gamma< 1$ very close to $0$, by
$$C_{u,e}^\gamma=\{v\in V_e:\|v^s\|\leq\gamma\|v^u\|\},\ \ \text{and}\ \ C_{s,e}^\gamma=\{v\in V_e:\|v^u\|\leq\gamma\|v^s\|\}.$$

Suppose $b$ is a recurrent point of $f$, i.e., there is a subsequence $\{b_{n_k}\}$ that converges to $b$. We take a large $N>0$ such that $f^N(b)$ lands in a trivialized neighborhood that contains the fibre $p^{-1}b$.

For each $e$ denote $\varphi(b,\psi(e'_N))=:e_N$, where $e'_N$ denotes $F^N_b(e)\in p^{-1}(f^Nb)$, that is, $e_N$ is the point in $p^{-1}b$ that has the same projection to the fibre as $e_N'$. Let $\varphi_b:\{b\}\times\T^d\to p^{-1}b$ denote the restriction of the chart to the fibre over $b$. Let $\bar{F}_{b}^N:=\varphi_b\circ\psi\circ F_{b}^N:p^{-1}b\to p^{-1}b$ and $\bar{F}_{b}^{-N}:=F_{b}^{-N}\circ \varphi_{f^N(b)}\circ \psi:p^{-1}b\to p^{-1}b$. Then $d\bar{F}_{b}^N$ and $d\bar{F}_{b}^{-N}$ again preserve the tangent spaces. For simplicity of the notation, we will also use $d\bar{F}^N_e:=(d\bar{F}^N_{b})_e:V_e\to V_{e_N}$ when we want to emphasize the point $e$.

That the cones are invariant under $dF_e^N:V_e\to V_{e_N'}$ is obvious. We claim that the cones are also invariant under $d\bar F^N_e: V_e\to V_{e_N}$.

\begin{lemma}\label{cones}\begin{itemize}\item[(1)] If $N$ is big enough, we have $$dF_e^{N}C_{u,e}^\gamma\subset\interior C_{u,e_N'}^{\gamma/2},\ \ \text{and} \ \ \ dF_e^{-N}C_{s,e_N'}^\gamma\subset\interior C_{s,e}^{\gamma/2};$$
\item[(2)] If $b'$ is close enough to $b$, then for all points $e\in p^{-1}b$ and $e'\in p^{-1}b'$,
$$ d\varphi_b\circ d\psi C_{u,e'}^{\gamma/2} \subset\interior C_{u,e}^{\gamma},\ \ \text{and}\ \ \  d\varphi_{f^N(b)}\circ d\psi C_{s,e}^{\gamma/2} \subset\interior C_{s,e'}^{\gamma}.$$
\item[(3)] There exists a $\lambda'\in (0,1)$ such that $$\|d\bar F^{N}_{e}(v^s,v^u)\|\geq(\lambda')^{-N} \|(v^s,v^u)\|,\ \ \ \text{if}\ \ (v^s,v^u)\in C_{u,e}^\gamma,\ \ \text{and}$$
$$\|d\bar F^{-N}_{e}(v^s,v^u)\|\geq (\lambda')^{-N} \|(v^s,v^u)\|,\ \ \ \text{if}\ \ (v^s,v^u)\in  C_{s,e}^\gamma.$$
\end{itemize} 
\end{lemma}
\begin{remark} We want to point out that the situations are slightly different for the stable and unstable cones, because for $d\bar F^N_b$ we post-composed the projection to the fibre, while for $d\bar F^{-N}_{b}$ we pre-composed the projection, just to make sure that they are both self maps of the tangent spaces of the same fibre $p^{-1}b$. Thus note that we do not really need the $\gamma/2$ for the stable cone in the statement of (1) in the above lemma for our purpose, but we still make the statement in a symmetric manner. In addition, $d\bar F^N_b$ and $d\bar F_b^{-N}$ are inverses of each other.
\end{remark}
\begin{proof}[Proof of Lemma \ref{cones}] 
We show all items for the stable cones.

To prove (1): Suppose $v^s+v^u\in C_{s,e'_N}^\gamma $. Then if $N>\frac{-\ln2}{2\ln\lambda}$,
\begin{align*} \|dF_{e}^{-N}v^u\|\leq \lambda^N\|v^u\|\leq \gamma\lambda^N\|v^s\|\leq \gamma \lambda^{2N}\|dF^{-N}_e v^s\|< \gamma/2\|dF^{-N}_e v^s\|.
\end{align*}
So $dF^{-N}_e(v^s+v^u)\in\interior C_{s,e}^{\gamma/2}$.

To prove (2): Recall that we picked the basis $p_k,\ k=1,\ldots,d$ for the vertical bundle, where $p_i,\ i=1,\ldots, l$ is a basis of $E^s$ and $p_i,\ i=l+1,\ldots, d$ is a basis of $E^u$, in the entire trivializing neighborhood of $e$, varying continuously. We want to show that if $\|a^i_ep_{i,e}\|\leq\gamma/2\|a^j_ep_{j,e}\|$ for $l+1\leq i\leq d$ and $1\leq j\leq l$, then with $d\varphi_{f^N(b)}\circ d\psi (a_e^kp_{k,e})=  (a^i_{e'}p_{i,e'}, a^j_{e'}p_{j,e'})$, $1\leq k\leq d, l+1\leq i\leq d, 1\leq j\leq l$ under the change of basis, we have $\|a^i_{e'}p_{i,e'}\|\leq\gamma\|a^j_{e'}p_{j,e'}\|$.

This is simply because of the continuity of the basis and compactness of the space. We give a detailed computation below. If we think of the basis vectors as constant vectors of $\R^d$ and denote the change of basis matrix $S_i^{\ j}$ such that $p_{i,e'}=S_i^{\ j}p_{j,e}$, $i,j=1,\ldots,d$. For any $\varepsilon>0$, there exists a uniform $\delta>0$ such that if $B_\delta$ is a ball of radius $\delta$ and $e,e'\in B_\delta$, then both $\|S_i^{\ j}-\id\|\leq \varepsilon/d^2$ and $\|(S_i^{\ j})^{-1}-\id\|\leq \varepsilon/d^2$. We could also assume $B_\delta$ is contained in a trivialized neighborhood by compactness of the base, and then by the tube lemma (we emphasize the compactness of the fibre here), we can find product neighborhoods that covers the entire bundle and whenever $e,e'$ are in the same neighborhood, both $\|S^{\ j}_i-\id\|\leq \varepsilon/d^2$ and $\|(S_i^{\ j})^{-1}-\id\|\leq \varepsilon/d^2$.

Now $l+1\leq i\leq d$, $1\leq j\leq l$, $1\leq k\leq d$, let $S_{\ j}^i$ denote $(S_i^{\ j})^{-1}$, and for simplicity let $a^k_{e'}p_{k,e'}= d\varphi_{f^N(b)}\circ d\psi (a_e^kp_{k,e})$  be a unit vector, where $a_e^kp_{k,e}\in  C_{s,e}^{\gamma/2}$ (note then $\|a^i_{e'}p_{i,e'}\|\leq 1$). We then have,
\begin{align*} \|a^i_{e'}p_{i,e'}\| &= \|a^k_{e} p_{k,e}\|= \|a^i_{e'} S_i^{\ k} S_{\ k}^i p_{i,e'}\| \ \ \ \ \  \text{(Write in different bases.)}\\
&\leq \|a^t_{e'} S_t^{\ i} S_{\ i}^t p_{t,e'}\|+\|a^t_{e'} S_t^{\ j} S_{\ j}^t p_{t,e'}\|\ \ \ \ \  \text{(Here $l+1\leq t\leq d,\ 1\leq j\leq l$.)} \\
& = \|a_e^ip_{i,e}\|+\|a^t_{e'} S_t^{\ j} p_{j,e'}\| \ \ \ \ \  \text{(Note that $a^k_{e}p_{k,e}\in C_{s,e}^{\gamma/2}$.)}\\
&\leq \gamma/2\|a^j_e p_{j,e}\|+ \frac{\varepsilon}{d^2}\cdot d^2  \ \ \ \ \  \text{(since $|S_t^{\ j}|\leq \frac{\varepsilon}{d^2}$ off diagonal.)}\\
&\leq \gamma/2 (\|a^j_{e'}p_{j,e'}\|+\varepsilon)+\varepsilon\ \ \ \ \ \text{(Repeat the above change of coordinates.)}\\
&\leq \gamma\|a^j_{e'}p_{j,e'}\|,\ \ \text{if we take sufficiently small $\varepsilon$.}
\end{align*}

To prove (3): Assuming (1) and (2), we have $\bar F^{-N}(e_N)=e$ and for any $(v^s,v^u)\in C_{s,e_N}^\gamma$,\ $d\bar F^{-N}(v^s,v^u)\in C_{s,e}^\gamma$. Then 
\begin{align*} \|d\bar F_{e}^{-N}(v^s,v^u)\| &\geq \|d\bar F^{N}_{e}v^s\| - \|d\bar F^{N}_{e}v^u\|\geq (1-\gamma) \|dF^{-N}_{e}v^s\|\geq (1-\gamma)\lambda^{-N}\|v^s\| \\
&\geq \frac{1-\gamma}{1+\gamma}\lambda^{N}\|(v^s,v^u)\|, \ \ \ \ \text{because $\|v^s\|\geq \frac{1}{1+\gamma}\|v^s+v^u\|$.}\end{align*}
We can take sufficiently small $\gamma$ and let $\lambda^N_1=\frac{1-\gamma}{1+\gamma}\lambda^N$; similarly for $C_{u,e}^\gamma$, we get a $\lambda_2$. Then we can take $\lambda'=\max\{\lambda_1,\lambda_2\}$.
\end{proof}

Following from the lemma and the cone criterion, $\bar{F}_b^N=\varphi_b\circ\psi\circ F_{b}^N$ is an Anosov diffeomorphism. Recall that $p^{-1}b$ is canonically identified with $\T^d$ up to a translation. Thus $\bar{F}_b^N$ induces a hyperbolic automorphsim $(\bar{F}_b^N)_*$ on $H_1(\T^d;\R)=\R^d$ \cite{manning1}. By continuity of $F$, we know that if $b,b'$ are close enough points in the base, $F_b$ and $F_{b'}$ are homotopic, and thus induce the same automorphism on $H_1(\T^d;\R)$. By compactness, we are able to extend this to the entire $B$, so for all $b\in B$, $(F_b)_*=A$ for a fixed $A$, which is also hyperbolic. We state this fact as the following lemma.

\begin{lemma} For any $b\in B$, the induced homomorphism $(F_b)_*: H_1(\T^d;\R)\to H_1(\T^d;\R)$ is the same hyperbolic automorphism.
\end{lemma}

Now suppose $F,G:E\to E$ are two arbitrary continuous maps that cover $f$, we can define $r:E\to\T^d$ to be such that $F(e)=G(e)+r(e)$. Note that morally $r$ is defined as the amount we need to translate from $G$ to get $F$. Because the action is free, $r$ is well-defined and unique. We show that it is continuous.

\begin{lemma} \label{cts} For two continuous maps $F,G:E\to E$ on a principal bundle with fibre $H$ that cover the same map, we have a continuous map $\bar{r}:E\to H$ such that $F(e).\bar{r}(e)=G(e)$.
\end{lemma}
\begin{proof} We use charts. Suppose for a fixed $b$, $f(b)\in U_j$. Define $r_{j}=(\psi_jF)^{-1}.\psi_jG:E\to H$. We show $r_j$ is independent of charts.

This is because if, at the same time, $f(b)\in U_k$ for some $k\neq j$, then $\psi_k=h_{jk}.\psi_j$. We have $r_j=r_k$. These $r_j$'s are continuous as compositions of continuous functions. Thus we get a globally defined continuous function $\bar{r}:E\to H$ that agrees with the $r_j$'s everywhere.
\end{proof}

Now suppose we have a fibre-preserving map $F:E\to E$ that induces the same matrix $A$ on the homology of every fibre. From \cite{GORH}, if $E$ admits an $A$-map (we will show it does), then we have the following commutative diagram. Here $A(E)$ means $E$ with cocycles $\{A\circ g_{ji}\}$ as in the notation of Property \ref{cocycle}.  For the existence and well-definedness of such bundles and $A$-maps, we refer to Proposition 4.6 and Theorem 6.2 in the same paper.
\begin{equation}\label{diagram}\begin{tikzcd} E\rar["F"]\dar & E\rar["\id_{\T^d}\text{-map}"]\dar & f^*E \rar["\id_{\T^d}\text{-map}"]\dar & A(E) \rar["A^{-1}\text{-map}"]\dar & E\dar\\
B\rar["f"] & B\rar["f^{-1}"] & B\rar["\id_B"] & B\rar["\id_B"] & B
\end{tikzcd}\end{equation}

Denote the composition in the upper row $\bar{F}$. Then the restricton of $\bar{F}$ to each fibre clearly induces the identity on $H_1(\T^d;\R)$.

\begin{lemma} Let $F:E\to E$ be a fibre-preserving map that induces the same automorphism $A$, which is not necessarily hyperbolic, on the homology of each fibre. Let $r:E\to\T^d$ be such that $\bar{F}=\id_E+r$, where $\bar{F}$ is as the above. Then $r$ is homotopic to $v\circ p$, where $v:B\to\T^d$ is a translation in the fibre and $p:E\to B$ is the projection, i.e., the following diagram commutes up to homotopy.
\begin{equation}\begin{tikzcd} E\rar["r"]\dar["p"'] & \T^d\\
B\arrow[ur,"v"'] \end{tikzcd}\end{equation}
\end{lemma}
\begin{proof} We have the following diagram.
\begin{equation}\begin{tikzcd}\pi_1(\T^d)\rar["i_\#",hook]\arrow[dr,"\text{trivial map}"',dashed] & \pi_1(E)\rar["p_\#"]\dar["r_\#" description] & \pi_1(B)\arrow[dl,dashed,"v_\#"]\rar & 1\\
& \pi_1(\T^d) & &
\end{tikzcd}\end{equation}
The upper row comes from the long exact sequence. The trivial map on the left comes from the fact that the restriction of $r$ to each fibre induces the zero map on $H_1(\T^d;\R)$.

Our goal is to define $v_\#$. Once it is defined, from the $K(\pi,1)$-spaces fact (see, for example, Hatcher \cite{hatcher}, 1B.9), we get a unique $v$, up to homotopy, from any homomorphism $\pi_1(B)\to\pi_1(\T^d)$.

For $\alpha\in\pi_1(B)$, because $p_\#$ is surjective, there is a $\beta\in\pi_1(E)$ such that $p_\#\beta=\alpha$. We define $v_\#\alpha=r_\#\beta$.

To show this is well-defined, suppose there is another $\beta'$ such that $p_\#\beta'=p_\#\beta=\alpha$, we need to show that $r_\#\beta'=r_\#\beta$.

We know $p_\#(\beta'\beta^{-1})=1$ so $\beta'\beta^{-1}\in\ker p_\#=\ima i_\#$. There is a $\gamma\in\pi_1(\T^d)$ such that $i_\#\gamma=\beta'\beta^{-1}$. But $r_\#i_\#\gamma=r_\#\beta'\beta^{-1}=1$.
\end{proof}

\begin{lemma} Let $G:E\to E$ be a map that induces the identity $\id_*$ on $H_1(\T^d)$. Then $E$ admits an $\id_{\T^d}$-map.
\end{lemma}
\begin{proof} We prove the statement by induction on $k$ for principal torus bundles over $\sph^k$. After showing for spheres, a general statement is true because manifolds have CW approximations and we can just do inductions on the dimension of skeletons.

For $k=1$, the bundle must be trivial, so there is trivially an $\id_{\T^d}$-map.

Now suppose the statement for the bundle over $\sph^k$ is true. Then for a bundle over $\sph^{k+1}$, we consider the clutching construction. The bundle splits into two trivial bundles over discs $\D^+$ and $\D^-$ of dimension $k+1$. The restrictions of the torus bundle to the boundary spheres $\sph^k$ of $\D^{+/-}$ each admits an $\id_{\T^d}$-map.

Let us denote the $\id_{\T^d}$-map $g:\sph^{k}\times\T^d\to\sph^{k}\times\T^d$. Parametrize the disc $\D^{k+1}$ by $(\rho,\theta)$, where $\rho$ gives the radius and $\theta$ is a $k$-dimensional vector of angles. In this coordinate, $\rho =1$ is the boundary shpere, and $g(\theta,y)=(\theta,\id+v(\theta))$ on the boundary. The problem actually reduces to the extension of $v(\theta)$ to the entire disc.

However, because the pair $(\D^{k+1}\times\T^d,\sph^{k}\times\T^d)$ is a good pair, which has the homotopy extension property, we are able to find a homotopy $h:\D\times\T^d\times I\to \D\times\T^d$ such that $h(1,\theta,y,0)=g(\theta,y)$. Thus $V(\rho,\theta)=h(\rho,\theta,0,0)$ is a map that restricts to $v(\theta)$ on the boundary, where $y=0$ is just the zero section of the trivial bundle. Now $\id_{\D\times\T^d}+V$ is an $\id_{\T^d}$-map on half of the bundle. We can glue via the clutching map.
\end{proof}

As a consequence, we have the following proposition.

\begin{proposition}\label{homot} If $F:E\to E$ is a fibre-preserving map that induces the same automorphism $A$ on the homology of each fibre, then $F$ is homotopic to an $A$-map.
\end{proposition}
\begin{proof} From the previous lemmas, we have found a homotopy $\bar{H}:E\times I\to\T^d$ such that $\bar{H}(e,0)=r(e)$ where $\bar{F}(e)=\id_E(e)+r(e)$, and $\bar{H}(e,1)=v\circ p(e)$. Then if we define $\tilde{H}:E\times I\to E$ as $\tilde H(e,t)=\id_E(e)+\bar H(e,t)$, it gives a homotopy such that $\tilde{H}(e,0)=\bar{F}(e)$ and $\tilde{H}(e,1)=\id_E(e)+v\circ p(e)$.

If we name the maps in diagram (\ref{diagram}), from the left to right, the first $\id_{\T^d}$-map $g_1$, the second $\id_{\T^d}$-map $g_2$, and the $A^{-1}$-map $g_3$, then $H:=g_1^{-1}g_2^{-1}g_3^{-1}\tilde{H}:E\times I\to E$ gives a homotopy such that $H(e,0)=F(e)$ and $H(e,1)=G(e)$ for some $A$-map $G$.
\end{proof}

Then Proposition \ref{mprop} follows.

\begin{proof}[Proof of Proposition \ref{mprop}] By the remarks we made right after the proof of Lemma \ref{cones}, if $F:E\to E$ is fibrewise Anosov, then $F$ induces the same hyperbolic automorphism $A$ on the homology of the fibre $\T^d$. Therefore $F$ is homotopic to $G$, for some fibrewise affine Anosov diffeomorphism $G$ from Proposition \ref{homot}.
\end{proof}

\section{Proof of Theorem \ref{main}}\label{pofmain}

From the last section we get a homotopy $H$ between $F,G:E\to E$ where $G$ is a fibrewise affine Anosov diffeomorphism. First we construct the $h$ as stated in Theorem \ref{main}. This is very similar to \cite[Proposition 2.1]{franks}.

For any two maps $F,G: E\to E$ on a principal torus bundle that covers the same map $f: B\to B$, again we let $r:E\to\T^d$ be the map such that $F(e)+r(e)=G(e)$. If $F\simeq G$ and we let $H:E\times I\to E$ denote the homotopy such that $H(e,0)=F(e)$ and $H(e,1)=G(e)$, then we can let $R:E\times I\to \T^d$ be the map such that $H(e,t)+R(e,t)=G(e)$. This is well-defined, and continuous in $e$ also by Lemma \ref{cts} and in $t$ because $H$ is continuous. We have $R(e,0)=r(e)$ and $R(e,1)=c$ where $c:E\to \T^d$ denotes a constant map.

Thus in the case where $F$ is homotopic to $G$, $r$ is nullhomotopic and it lifts to a map $\tilde{r}:E\to\R^d$. Now if $G$ is fibrewise affine Anosov with the matrix $A$, we can write $\tilde{r}=\tilde{r}^s+\tilde{r}^u$ where $\tilde{r}^s$ and $\tilde{r}^u$ take values in the stable and unstable subspaces of $A$, respectively.

\begin{proposition}\label{semii} Suppose $G:E\to E$ is a fibrewise affine Anosov diffeomorphism with the matrix $A$ as specified in the definition. For any $F:E\to E$ that is homotopic to $G$ and covers the same map on the base with $G$, there is a continuous surjective map $h:E\to E$ homotopic to $\id_E$ which fibres over $\id_B$ such that $h\circ F=G\circ h$.
\end{proposition}
\begin{proof} We define $\tilde w^s, \tilde w^u:E\to E$,
$$\tilde{w}^s(e)=-\sum_{n=0}^\infty A^n\tilde{r}^s(F^{-(n+1)}(e)),\ \ \text{and}\ \ \ \tilde{w}^u(e)=\sum_{n=0}^\infty A^{-(n+1)}\tilde{r}^u(F^{n}(e)).$$
Then we have
$$ A\tilde{w}^s(e)-\tilde{w}^s(F(e))
=-\sum_{n=0}^\infty A^{n+1}\tilde{r}^s(F^{-(n+1)}(e))+\sum_{n=0}^\infty A^{n}\tilde{r}^s(F^{-n}(e))
=\tilde{r}^s(e).$$
Similarly we can check that $A\tilde{w}^u(e)-\tilde{w}^u(F(e))=\tilde{r}^u(e)$. Let $\tilde{w}=\tilde{w}^s\oplus\tilde{w}^u$. Then we have $A\tilde{w}(e)-\tilde{w}(F(e))=\tilde{r}(e)$.

We show that $\tilde{w}:E\to E$ is continuous. Let $\R^d$ be endowed with the Riemannian metric lifted from $\T^d$. Let $d_E$ denote the metric on $E$ and $d_{\R^d}$ the metric on $\R^d$. Since $E$ is assumed to be compact, there is a uniform bound $M$ of the distance from $\tilde{r}(e)$ to $0\in\R^d$ for all $e\in E$. For a given $\varepsilon>0$, take $N\in\N$ so we have $\sum_{n>N}\lambda^nM<\varepsilon/4$. Then since both $-\sum_{n=0}^NA^n\tilde{r}^s(F^{-(n+1)}(e))$ and $\sum_{n=0}^N A^{-(n+1)}\tilde{r}^u(F^{n}(e))$ are uniformly continuous, there is a $\delta>0$ such that if $d_E(e_1,e_2)<\delta$ for any $e_1,e_2\in E$, we have both $$d_{\R^d}\Big(\sum_{n=0}^N A^n\tilde{r}^s(F^{-(n+1)}(e_1)),\sum_{n=0}^N A^n\tilde{r}^s(F^{-(n+1)}(e_2))\Big)<\varepsilon/4$$
and
$$d_{\R^d}\Big(\sum_{n=0}^N A^{-(n+1)}\tilde{r}^u(F^{n}(e_1)),\sum_{n=0}^N A^{-(n+1)}\tilde{r}^u(F^{n}(e_2))\Big)<\varepsilon/4$$
Then $d_{\R^d}(\tilde{w}(e_1),\tilde{w}(e_2))<\varepsilon$.

We can then project $\tilde{w}:E\to\R^d$ and get $p_1\tilde{w}=w:E\to\T^d$, where $p_1:\R^d\to\T^d$ denotes the covering projection. Note $p_1:\R^n\to\T^n$ is linear. Then
$$Aw(e)-w(F(e))=Ap_1\tilde{w}(e)-p_1\tilde{w}(F(e))=p_1A\tilde{w}(e)-p_1\tilde{w}(F(e))=p_1\tilde{r}(e)=r(e).$$

Let $h: E\to E$ be such that $h(e)=e+w(e)$, which obviously covers the identity. We have
\begin{align*} G\circ h(e)&=G(e+w(e))=G(e)+Aw(e)
=G(e)+r(e)-r(e)+Aw(e)\\
&=F(e)+w(F(e))=h\circ F(e).
\end{align*}
Since $\tilde{w}$ is bounded, it is homotopic to a constant map $E\to\R^d$, and thus $h$ is homotopic to the identity via a homotopy that preserves each fibre. Then $h$ induces a map of nonzero degree on the top homology, and thus it is surjective.
\end{proof}

We denote $\tilde{F}_b:\R_b^d\to\R^d_{f(b)}$ the lift of the restriction $F_b:\T^d_b\to\T_{f(b)}^d$. We use $\widetilde{W}^s_{\tilde{F}_b}(\tilde{e})$ and $\widetilde{W}^u_{\tilde{F}_b}(\tilde{e})$ to denote the stable and unstable leaves of $\tilde{F}_b$ passing through a point $\tilde{e}$ in $\R^d_b$. We will also use $\widetilde{W}^s_{\tilde{F}_b}(\tilde{F}_b^n(\tilde{e}))$ and $\widetilde{W}^u_{\tilde{F}_b}(\tilde{F}_b^n(\tilde{e}))$ to denote the leaves along an orbit of $b$ lifted fibrewise. If it is clear that the leaves of which map we are talking about, we would also write $\widetilde W_b^s(\tilde e)$ and $\widetilde W_b^u(\tilde e)$.

Let $d(s;\cdot,\cdot)$ and $d(u;\cdot,\cdot)$ denote the distance between two points along stable and unstable leaves respectively.

If any pair of stable and unstable leaves of $\tilde{F}_b$ has a unique intersection, we say that we have the global product structure (GPS) of the leaves of $\tilde{F}_b$ in the fibre over $b$.

We want to show that the global product structure at all $b\in B$ in order to prove the injectivity of $h$ (see Lemma \ref{hinj}). Before that we check the following properties of either $h$ or the fibrewise lift $\tilde h_b:\R^d_b\to\R^d_b$, which we will use repeatedly.

\begin{property}\label{hpreservleaves} The restriction $h_b$ of $h$ to a fibre maps a stable (or an unstable) leaf of $F_b$ to a stable (or an unstable) leaf of $G_b$.
\end{property}
\begin{proof} It is sufficient to show that $h_b$ takes a local stable disk to a local stable disk, i.e. for any $e\in\T^d_b$, if $e' \in W^s_{F_b,\varepsilon}(e)$, we will have $h_b(e')\in W^s_{G_b,\varepsilon}(h_b(e))$, for some $\varepsilon>0$.

Suppose the statement is not true, say $h_b(e')\notin W^s_{G_b,\varepsilon}(h_b(e))$. Then
$$d_{\T_{f^n(b)}^d}(G^n_b(h_b(e)),G^n_b(h_b(e')))=d_{\T_{f^n(b)}^d}(h_{f^n(b)}F^n_b(e),h_{f^n(b)}F^n_b(e'))\to 0,$$
as $n\to\infty$. But the $G_b$ is the fibrewise affine Anosov diffeomorphism which has straight leaves in the eigen-direction of $A$, so the left hand side of the equation goes to infinity. We have reached a contradiction.
\end{proof}

\begin{remark} From the the above we also have that the lift $\tilde h_b$ takes a stable (or an unstable) leaf of $\tilde F_b$ to a stable (or an unstable) leaf of $\tilde G_b$.
\end{remark}

\begin{property}[(\cite{franks})]\label{hproper} Since $\tilde h_b$ is homotopic to the identity along the fibres, it is proper.
\end{property}

We claim that the existence of an intersection for any pair of stable and unstable leaves in the fibrewise covering spaces follows exactly from Lemma 1.6 of Franks \cite{franks1}, which we include for completeness. Since the proof there is rather long, we only include a sketch. Note that we are only going to use the local product structure and the semiconjugacy. This works for any point $b\in B$.

\begin{lemma} For any points $\tilde e,\tilde e'\in \R_b^d$, we have $\widetilde W^s_{\tilde F_b}(\tilde e)\cap \widetilde W^u_{\tilde F_b}(\tilde e')\neq\emptyset$ and $\widetilde W^u_{\tilde F_b}(\tilde e)\cap \widetilde W^s_{\tilde F_b}(\tilde e')\neq\emptyset$.
\end{lemma}
\begin{proof}[Sketch of proof] Fix a stable leaf $S$ of $\tilde F_b$, and we want to show that the set
$$Q:=\{\tilde e\ |\ \widetilde W^u_{\tilde F_b}(\tilde e)\cap S\neq\emptyset\}$$
is $\R^d_b$. Because of the local product structure, there is an open set $O$ that contains $S$ such that $Q=\{\tilde e\ |\ \widetilde W^u_{\tilde F_b}(\tilde e)\cap O\neq\emptyset\}$. If $\tilde e\in Q$, there is a foliated neighborhood $U$ containing $\tilde e$ which we could extend to a path of foliated neighborhoods whose union contains $\widetilde W^u_{\tilde F_b}(\tilde e)$, so $U\subset Q$. Thus $Q$ is open. We want to show that $Q$ is also closed.

Suppose $x$ is a point in the closure of $Q$ but not in $Q$. We show that $\widetilde W^u_{\tilde F_b}(x)\cap S\neq\emptyset$. Because of the local product structure, we are able to find a sequence of points in $Q$ approaching $x$ which is also contained in the same stable leaf $\widetilde W^s_{\tilde F_b}(x)$. We pick one point from the sequence and denote it by $y$. Pick another point $z$ in $\widetilde W^u_{\tilde F_b}(y)\cap S$. We connect $y$ and $z$ by a path $\gamma:[0,1]\to\widetilde W^u_{\tilde F_b}(y)$ and connect $x$ and $y$ by a path $\rho:[0,1]\to\widetilde W^s_{\tilde F_b}(x)=\widetilde W^s_{\tilde F_b}(y)$.

We aim to define a function $\theta:[0,1]\times [0,1]\to\R_b^d$ such that $\theta(r,0)=\gamma(r)$ and $\theta(0,t)=\rho(t)$, and $\theta(r,t)=\widetilde W^s_{\tilde F_b}(\gamma(r))\cap \widetilde W^u_{\tilde F_b}(\rho(t))$. If we are able to define $\theta$ at $(1,1)$, then $\widetilde W^s_{\tilde F_b}(z)\cap \widetilde W^u_{\tilde F_b}(x)=\widetilde W_{\tilde F_b}^u(x)\cap S\neq\emptyset$.
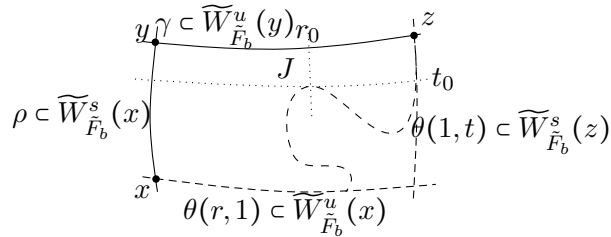
\begin{figure}[H]
    \centering
    \begin{tikzpicture}
  %left
  \draw plot[smooth, tension=0.7] coordinates{(0,2.2) (-0.1,1) (0,0)};
  \node at (-1,1) {$\rho\subset\widetilde W^s_{\tilde F_b}(x)$};
  %top
  \draw plot[smooth, tension=0.8] coordinates{(-.2,2) (1.7,1.9) (3.5,2.1)};
  \node at (.9,2.2) {$\gamma\subset\widetilde W^u_{\tilde F_b}(y)$};
  %right
  \draw[densely dashed] plot[smooth, tension=0.7] coordinates{(3.4,2.3) (3.45,1.2) (3.4,-.1)};
  \node at (4.7,.8) {$\theta(1,t)\subset\widetilde W^s_{\tilde F_b}(z)$};
  %bottom
  \draw[densely dashed] plot[smooth, tension=0.7] coordinates{(-.2,.2) (1.8,0) (3.7,.1)};
  \node at (1.7,-.3) {$\theta(r,1)\subset \widetilde W^u_{\tilde F_b}(x)$};
  %points
    \node at (-.02,.17)[circle,fill,inner sep=1pt]{};
    \node at (-.2,0) {$x$};
    \node at (-.03,1.99)[circle,fill,inner sep=1pt]{};
    \node at (-.2,2.1) {$y$};
    \node at (3.41,2.08)[circle,fill,inner sep=1pt]{};
    \node at (3.6,2.3) {$z$};
    %where theta is defined
    \draw[dashed] plot[smooth, tension=0.9] coordinates{(3.41,2.08) (3.2,.8) (2,1.4) (1.8,.5) (2.5,.3) (2.5,0)};
    %t_0, r_0
    \draw[dotted] plot[smooth,tension=.9] coordinates{(-.3,1.5) (2,1.4) (3.6, 1.5)};
    \node at (3.8,1.5) {$t_0$};
    \draw[dotted] plot[smooth,tension=.9] coordinates{(2,2.2) (2.05,1)};
    \node at (2,2.1) {$r_0$};
    \node at (1.7,1.65) {$J$};
\end{tikzpicture}
\caption{Definition of $\theta$} \label{figexistence}
\end{figure}
Because of the local product structure, the set where $\theta$ can be defined is open. Now let
$$t_0=\sup\{\hat t\ |\ \theta(r,t) \text{ is defined for } 0 \leq r\leq 1\text{ and } 0\leq t\leq \hat t\},\ \text{and}$$
$$r_0=\sup\{\hat r\ |\ \theta(r,t_0)\text{ is defined for } 0\leq r\leq\hat r\}.$$
If we let $J:=\{(r,t)\ |\ t<t_0\}$, $\tilde h_b(\theta(J))$ is bounded because the projections of it to the linear leaves $\widetilde W^u_{\tilde G_b}(\tilde h_b(y))$ and $\widetilde W^s_{\tilde G_b}(\tilde h_b(y))$ are bounded. Then $\theta(J)$ is bounded because $\tilde h_b$ is proper. Next pick a sequence $\{(r_n,t_n)\ |\ r_n\leq r_{n+1}\}$ in $J$ converging to $(r_0,t_0)$ such that $\theta(r_n,t_n)$ converges to a point $w\in\R^d_b$. We define $\theta(r_0,t_0):=w$, by continuity. 
What is left is to check that $\theta$ is well-defined.

Let $N$ be a product neighborhood of $w$. We can assume $\theta(r_n,t_n)\in N$ for all $n>0$.
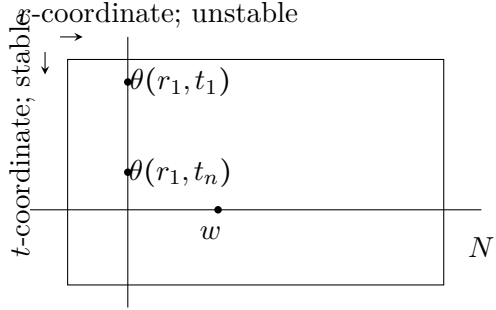
\begin{figure}[H]
    \centering
    \begin{tikzpicture}[>=stealth]
  \draw (0,0) rectangle (5,3);
    \node at (2,1)[circle,fill,inner sep=1pt]{};
    \node at (1.9,.7) {$w$};
    \node at (5.5,.5) {$N$};
    %leaves 
    \draw (-.5,1) -- (5.5,1);
    \draw (.8,3.3) -- (.8,-.3);
    \node at (.8,2.7)[circle,fill,inner sep=1pt]{};
    \node at (1.5,2.7) {$\theta(r_1,t_1)$};

    %points
    \node at (.8,1.5)[circle,fill,inner sep=1pt]{};
    \node at (1.5,1.5) {$\theta(r_1,t_n)$};
    \draw[->] (-.1,3.3) -- (.2,3.3);
    \draw[->] (-.3,3.1) -- (-.3,2.8);
    \node at (1.2,3.6) {$r$-coordinate; unstable};
    \node at (-.6,2) [rotate=90] {$t$-coordinate; stable};

\end{tikzpicture}
\caption{Extending $\theta$ to $(r_0,t_0)$} \label{theta}
\end{figure}
Since $\theta$ is defined at $(r_n,t_n)$ for $n>0$ and is by definition $\widetilde W^s_{\tilde F_b}(\gamma(r_n))\cap\widetilde W^u_{\tilde F_b}(\rho(t_n))$. Thus $\theta(r_1,t_n)$ lies on $\widetilde W^s_{\tilde F_b}(\gamma(r_1))$ for all $n>0$. Then $\theta(r_1,t_n)$ converges to $\widetilde W^s_{\tilde F_b}(\theta(r_1,t_1))\cap \widetilde W^u_{\tilde F_b}(w)$. (Suppose not. If $\theta(r_1,t_n)$ converges to $\bar w:=\widetilde W^s_{\tilde F_b}(\theta(r_1,t_1))\cap\widetilde W^u(w')$ for a point $w'\neq w$, then $\theta(r_n,t_n)$ would converge to a point inside $\widetilde W^u_{\tilde F_b}(\bar w)\neq \widetilde W^u_{\tilde F_b}(w)$.) Similarly, $\theta(r_n,t_1)$ converges to $\widetilde W^s(w)_{\tilde F_b}\cap\widetilde W^u_{\tilde F_b}(\theta(r_1,t_1))$. Thus it is either defined already so we have, or we could define that
$$\theta(t_1,r_0)=\widetilde W^s_{\tilde F_b}(\theta(r_1,t_1))\cap \widetilde W^u_{\tilde F_b}(w),\ \ \theta(t_0,r_1)=\widetilde W^s(w)_{\tilde F_b}\cap\widetilde W^u_{\tilde F_b}(\theta(r_1,t_1)).$$

Now suppose somehow we are able to define for another sequence $\{(r_n',t_n')\}$ that converges to $(r_0,t_0)$ such that $\theta(r_0,t_0)=\lim_{n\to\infty}\theta(r_n',t_n')$. Then we have
\begin{align*} \lim_{n\to\infty}\theta(r_n',t_n') &= \lim_{n\to\infty}\widetilde W^u_{\tilde F_b}(\theta(r_1,t_n'))\cap \widetilde W^s_{\tilde F_b}(\theta(r_n',t_1)) \\
 &=\widetilde W^u_{\tilde F_b}(\theta(r_1,t_0))\cap\widetilde W^s_{\tilde F_b}(\theta(r_0,t_1))=w.
\end{align*}
This is saying that we could always extend $\theta$ to a limit point, so $Q$ is closed.
\end{proof}

Next we are going to show the uniqueness of the intersection for each pair of stable and unstable leaves.

\begin{remark} Since we have already showed the existence of the intersections, if we take any pair of stable and unstable leaves, they must intersect, say in a point $\tilde e\in p^{-1}b$. It is then sufficient to show that this pair of stable and unstable leaves has no other intersections besides $\tilde e$.

On the other hand, for any point $\tilde e\in p^{-1}b$, there is a pair of stable and unstable leaves that intersects at $\tilde e$, namely, $\widetilde W^s_{\tilde F_b}(\tilde e)$ and $\widetilde W^u_{\tilde F_b}(\tilde e)$, which are exactly the pair of leaves that we considered in the above paragraph.

Therefore to show the unique intersection for an arbitrary pair of leaves, we only need to show that for any $\tilde e$, $\widetilde W^s_{\tilde F_b}(\tilde e)\cap \widetilde W^u_{\tilde F_b}(\tilde e)=\{\tilde e\}$.
\end{remark}

To summarize the steps of the proof:
\begin{itemize}\item[1. ] If $b$ is a recurrent point of $f$ in the base, then in $p^{-1}b$ any pair of stable and unstable leaves has a unique intersection. This is Lemma \ref{recgps}.
\item[2.] We show that the set
$$\{b\in B\ |\ \text{The foliations induced by $F_b$ has the global product structure.} \}$$
is open (with a uniform constant $\delta$ such that at every point where we already have GPS, any point in a ball with radius $<\delta$ has GPS). This is Corollary \ref{gpsopen}.
\item[3.] Corollary \ref{gpsopen} also implies that the above set is closed because the $\delta$ is uniform. This is Corollary \ref{gpsclosed}.
\end{itemize}

Now we proceed with the proof.

\begin{lemma}\label{recgps} Let $b\in B$ be a recurrent point of $f$. Then $\widetilde{W}^u_{\tilde{F}_b}(\tilde{e})\cap \widetilde{W}^s_{\tilde{F}_b}(\tilde{e})=\{\tilde e\}$.
\end{lemma}
\begin{proof} We use the fact that the leaves are ``very close'' to the leaves of an Anosov diffeomorphism.

Assume $\widetilde{W}^u_{\tilde{F}_b}(\tilde{e}),\widetilde{W}^s_{\tilde{F}_b}(\tilde{e})$ intersect in two points, say $\tilde{e},\tilde{e}'$. Then $\tilde{F}^n_b(\tilde{e}')\in \widetilde{W}^u_{\tilde{F}_{b}}(\tilde{F}^n_b(\tilde{e}))\cap \widetilde{W}^s_{\tilde{F}_{b}}(\tilde{F}^n_b(\tilde{e}))$, which means the leaves along the orbit also intersect twice.

Let $\{f^{n_k}(b)\}$ denote the subsequence that converges to $b$. We take an $L$ large enough so that $f^{n_L}(b)$ comes back close to $b$ where we have the invariant cones with the constant $\gamma$ as in Lemma \ref{cones}. 
Let $\bar{F}^{n_L}_b$ denote the Anosov diffeomorphism that is $C^1$ close to $\tilde{F}^{n_L}_b$ as before.

We can assume a lower bound, say $5\varepsilon$, for the constant of the local product structure of each of $\{\bar{F}_b^{n_k}\}$ while $k>L$, because of the continuity of the leaves.

Now we take $L'\geq L$ large enough so we can have a small enough $\gamma'$ such that
$$d(s;\tilde{F}_b^{n_{L'}}(\tilde{e}),\tilde{F}_b^{n_{L'}}(\tilde{e}'))\cdot \gamma' < \varepsilon.$$
Since the local basis of $\bar{E}^s$ at each point lies in the same $\gamma'$-cone of $\tilde{E}^s$, this means $\widetilde{W}^s_{\bar{F}_b}(\bar{F}_b^{n_{L'}}(\tilde{e}))$ lies in the normal neighborhood $N(\widetilde{W}^s_{\tilde{F}_b}(\tilde{F}_b^{n_{L'}}(\tilde{e})),\varepsilon)$.

Now the unstable leaf $\widetilde{W}^u_{\bar{F}_b}(\bar{F}_b^{n_{L'}}(\tilde{e}))$ must pass through $N(\widetilde{W}^s_{\tilde{F}_b}(\tilde{F}_b^{n_{L'}}(\tilde{e})),\varepsilon)$ near $\tilde{F}_b^{n_{L'}}(\tilde{e}')$ because of the existence of the unstable cone. Then $\widetilde{W}^s_{\bar{F}_b}(\bar{F}_b^{n_{L'}}(\tilde{e}))$ and $\widetilde{W}^u_{\bar{F}_b}(\bar{F}_b^{n_{L'}}(\tilde{e}))$ should intersect the second time because of the local product structure, which is not the case. Thus $\widetilde{W}^u_{\tilde{F}_b}(\tilde{e})$ and $\widetilde{W}^s_{\tilde{F}_b}(\tilde{e})$ cannot intersect more than once.
\end{proof}

Next we prove a useful technical lemma, which is implicit in \cite{franks} (Proof of Theorem 1).

%this is fine
\begin{lemma}\label{distleaf1} Let $\F_1,\ \F_2$ be a pair of transverse foliations on $\T^d$ with continuous leaves and the local product structure, such that its lift $\tilde\F_i,\ i=1,2$ also has the global product structure on $\R^d$. Let $d_1,\ d_2$ denote the distance along a leaf of $\tilde\F_1,\ \tilde\F_2$ between two points in the same leaf, respectively. Then for any compact set $K\subset\R^d$, there is a constant $M_K>0$ such that
$$\sup\big\{d_1(\tilde{e},\tilde{e}'): \tilde{e},\tilde{e}'\in (K\cap\widetilde{W}^1(\tilde{e})),\ \widetilde{W}^1\text{ is a leaf of }\tilde\F_1 \big\}<M_K,\ \text{and}$$
$$\sup\big\{d_2(\tilde{e},\tilde{e}'): \tilde{e},\tilde{e}'\in (K\cap\widetilde{W}^2(\tilde{e})),\ \widetilde{W}^2\text{ is a leaf of }\tilde\F_2 \big\}<M_K,$$
i.e., the distance along the leaves are bounded in a compact set.
\end{lemma}
\begin{proof} We show it for $\tilde\F_2$. From the global product structure, we get a homeomorphism $\alpha=(\alpha_1,\alpha_2):\R^d\to \R^l\times\R^{d-l}$, where $\alpha_1,\alpha_2$ are the projections, $\tilde\F_2$ is of dimension $d-l$ and codimension $l$, and the latter $\R^l\times\R^{d-l}$ has the usual topology.

Suppose there is not such an upperbound $M_K$ for the distance along leaves. Then there exists a sequence of pairs $\{(\tilde{e}_n,\tilde{e}_n')\}$ such that $\tilde{e}_n,\tilde{e}_n'\in K\cap \widetilde{W}^2(\tilde{e}_n)$ and $d_1(\tilde{e}_n,\tilde{e}_n')>n$.

By the compactness of $K$, there is a subsequence $\{(\tilde{e}_{n_k},\tilde{e}'_{n_k})\}$ that converges to a pair of points $(\tilde{e}_0,\tilde{e}_0')$. Since $\alpha_1(\tilde{e}_n)=\alpha_1(\tilde{e}'_n)$ (because they are on the same leaf of $\tilde \F_2$), we know $\alpha_1(\tilde{e}_0)=\alpha_1(\tilde{e}_0')$, that is, $\tilde{e}'_0\in \widetilde{W}^2(\tilde{e}_0)$.

Now we can cover the path between $\tilde{e}_0$ and $\tilde{e}'_0$ that realizes $d_2(\tilde{e}_0,\tilde{e}'_0)$ in $\widetilde{W}^2(\tilde{e}_0)$ with a finite number of cubes with length of edges smaller than the constant of the local product structure, and denote the cubes by $\{P_1,\ldots, P_t\}$. Then $\alpha (\cup_{i=1}^t P_i)$ covers a neighborhood $B_\varepsilon \times \alpha_2(\widetilde{W}^2(\tilde{e}_0)\cap K)$ where $B_\varepsilon$ is a small $l$-dimensional disk. Thus $\alpha^{-1}(B_\varepsilon \times \alpha_2(\widetilde{W}^2(\tilde{e}_0)\cap K))$ also covers a piece between $\tilde{e}_{n_k}$ and $\tilde{e}'_{n_k}$ of $\widetilde{W}^2(\tilde{e}_{n_k})$ for all $k>L$ for some large $L$, because of the local product structure. But $d_2(\tilde{e}_{n_k},\tilde{e}'_{n_k})>k$. They cannot be all covered by a finite number of cubes of a fixed diameter. We have reached a contradiction.
\end{proof}

\begin{corollary}\label{distleaf2} Let $C\subset B$ be a compact set contained in a trivialized neighborhood of the bundle we consider. Consider the leaves of $F$ lifted continuously over $C$. Suppose we have the global product structure at all $b\in C$. Then for any compact set $K\subset C\times\R^d$, there is an $M_C>0$ such that
$$\sup_{b\in C}\,\sup\big\{d(u;\tilde{e},\tilde{e}'): \tilde{e},\tilde{e}'\in (K\cap\widetilde{W}^u_{b}(\tilde{e}))\big\}<M_C,\ \text{and}$$
$$\sup_{b\in C}\,\sup\big\{d(s;\tilde{e},\tilde{e}'): \tilde{e},\tilde{e}'\in (K\cap\widetilde{W}^s_{b}(\tilde{e}))\big\}<M_C.$$
\end{corollary}
\begin{proof} We are able to run the above argument, starting with a homeomorphism $\alpha= id\times(\alpha_1,\alpha_2):C\times \R^d\to C\times (\R^l\times\R^{d-l})$ because of the global product structure.

If we want to check for the unstable foliation, then instead of $B_\varepsilon \times \alpha_2(\widetilde{W}^2(\tilde{e}_0)\cap K)$ we consider a neighborhood $A\times B_\varepsilon \times \alpha_2(\widetilde{W}^u_{b_0}(\tilde{e}_0)\cap K)$ covered by a finite number of cubes given by the local product strucutre, where $p^{-1}b_0$ contains $\tilde e_0,\tilde e_0'$, $A$ is a small neighborhood around $b_0$ and $B_\varepsilon$ is an $l$-dimensional disk. The rest of the argument are the same.
\end{proof}

%inj of lift?
\begin{lemma}\label{hinj} If we have the global product structure at a point $b\in B$, then the restriction $\tilde h_b:\R^d_b\to\R^d_b$ is injective.
\end{lemma}
\begin{proof} Note that if we have the global product structure at one point $b$, then we also have it along the orbit of $b$, if considering the lifts fibrewise. We are then able to follow a similar argument to Franks \cite{franks1} to deduce the injectivity of $\tilde h_b$.

Recall that $\tilde h_b$ preserves the stable and unstable leaves (Property \ref{hpreservleaves}). Suppose after fixing the lifts, we can find $\tilde{e}\neq\tilde{e}'$ such that $\tilde{h}_b(\tilde{e})=\tilde{h}_b(\tilde{e}')$. Let $\tilde{e}'':=\widetilde{W}^u_{\tilde{F}_b}(\tilde{e})\cap \widetilde{W}^s_{\tilde{F}_b}(\tilde{e}')$. We know $\widetilde{W}^u_{\tilde{G}_b}(\tilde{h}_b (\tilde{e})) \cap \widetilde{W}^s_{\tilde{G}_b}(\tilde{h}_b (\tilde{e}'))=\tilde{h}_b(\tilde{e}'')$, i.e., they also intersect only once in $\tilde{h}_b(\tilde{e}'')$. Then $\tilde{h}_b(\tilde{e})=\tilde{h}_b(\tilde{e}')=\tilde{h}_b(\tilde{e}'')$. Thus we have found $\tilde{e},\tilde{e}''\in \widetilde{W}^u_{\tilde{F}_b}(\tilde{e})$ such that $\tilde{h}_b(\tilde{e})=\tilde{h}_b(\tilde{e}'')$.

Since $B$ is assumed to be a compact manifold (which is normal), for each finite open cover we can find a finite refinement that contains compact sets. Now we use compact trivialized neighborhoods, and $h_b$ is lifted continuously over each of them, denoted as $\tilde{h}_C$ on a compact set $C$. Then $\tilde{h}_C$ is homotopic to the identity along the fibres, and thus is proper (Property \ref{hproper}). Denote the set $(\id,\tilde{h}_C)^{-1}(C\times[0,1]^d)$ as $D$, which is compact.

% Because the set $B'\subset B$ of all points where we have the global product structure is compact, taking intersection with a compact trivialized neighborhood $C$, we still get a compact set $C'$, where we are able to lift $h$ continuously and then get a $D':=(\id,\tilde{h}_{C'})^{-1}(C'\times[0,1]^d)$, which is also compact.

Then by Corollary \ref{distleaf2} we have an $M_{C}>0$ such that
$$\sup_{b\in C}\,\sup\big\{d(u;\tilde{e},\tilde{e}''): \tilde{e},\tilde{e}''\in (\tilde{h}_b^{-1} [0,1]^d\cap \widetilde{W}^u_{\tilde{F}_b}(\tilde{e}))\big\}<M_{C}.$$
Then again by the compactness of $B$ there is a uniform bound $M>0$ such that
$$\sup_{b\in B}\,\sup\big\{d(u;\tilde{e},\tilde{e}''): \tilde{e},\tilde{e}''\in (\tilde{h}_b^{-1} [0,1]^d\cap \widetilde{W}^u_{\tilde{F}_b}(\tilde{e}))\big\}<M.$$

Denote $\tilde{w}_n:=\tilde{h}_{f^n(b)}\tilde{F}_{b}^n(\tilde{e})=\tilde{h}_{f^n(b)}\tilde{F}_{b}^n(\tilde{e}'')$. We can always use deck transformations to translate the $\tilde{w}_n$'s to $[0,1]^d$. Thus we should also have $d(u;\tilde{F}_{b}^n(\tilde{e}),\tilde{F}_{b}^n(\tilde{e}''))<M$ for all $n$. We have reached a contradiction.
\end{proof}

Note again the above shows that if at $b$ we have the global product structure, $\tilde{h}_b$ is a homeomorphism. It is surjective for free because it has nonzero degree on the top homology. In particular, $\tilde h_b$ is a homeomorphism for all recurrent points $b\in B$ now.

Then we show that the property of having the global product structure is open.

\begin{remark} Our main idea of the next proof is just to approximate the leaves where we do not know if they have the global product structure, with the leaves that already have the global product structure, with the existence of the cones.

However, since the universal cover is noncompact, we could quickly lose control of the error. Thus we emphasize the use of the semiconjugacy (now already a diffeomorphism at a nearby point where the GPS exists), which is bounded from the identity, so the nonlinear leaves is of a bounded distance from the linear leaves of our fibrewise affine Anosov diffeomorphism, which is essential to make the approximation work.
\end{remark}

\begin{proposition}\label{fol} Let $\F_1,\F_2$ be a pair of transverse foliations on $\T^d$ with $C^1$ leaves and the local product structure. Suppose the lifts $\tilde\F_i$ of $\F_i,\ i=1,2$ have the global product structure on $\R^d$. Let $\F_1'$ and $\F_2'$ be another pair of transverse foliations.

Then there exists an $\alpha>0$ such that,  if
\begin{itemize}
    \item[1.] $\max_x\angle(T\F_i(x),T\F_i'(x))<\alpha$ for $i=1,2$;
    \item[2.] there is a homeomorphism $h:\T^d\to\T^d$ which is homotopic to the identity and takes $\F_i,\ i=1,2$ to another pair of linear foliations $\F_i'',\ i=1,2$,
\end{itemize}
then the pair of lifts $\tilde\F'_i$ of $\tilde\F_i',\ i=1,2$ has the global product structure.
\end{proposition}

\begin{remark} We tried to state this in a more general manner. In our setting, the $\F_i$'s are the foliations of the fibrewise Anosov diffeomorphism at a point, say $b$, where we do not have the GPS yet, the $\F'_i$'s are the foliations of $F$ at a point that is close to $b$ where we have the GPS, and $\F''_i$'s are the foliations of $G$ with straight leaves.

Note that this is purely a statement about the foliations which has nothing to do with the dynamics.
\end{remark}

\begin{proof}[Proof of Proposition \ref{fol}] We first fix the notations. We will denote leaves at $\tilde e$ of $\tilde\F_i$ by $\widetilde W_1^i(\tilde e)$ and leaves of $\tilde\F_i'$ by $\widetilde W_2^i(\tilde e)$, $i=1,2$. Then for any $\tilde e_0\in \R^d$, we want to show that $\widetilde W^1_2(\tilde e_0)\cap\widetilde W^2_2(\tilde e_0)=\{\tilde e_0\}$. We start from selecting constants:
\begin{itemize}
    \item $\varepsilon$: Let $\varepsilon>0$ be smaller than the constant of the local product structure.
    \item $M_1$: Let $M_1$ be the constant such that $\|h-\id\|< M_1$.
    \item $D$: We fix a $D>10(M_1+\varepsilon)$.
    \item $M_2$: Suppose $K$ is a compact set of diameter $\leq D$. There is an $M_2$ such that
    $$\sup\big\{d_1(\tilde{e},\tilde{e}'): \tilde{e},\tilde{e}'\in (K\cap \widetilde{W}^1_1(\tilde{e})),\ \widetilde{W}^1_1\text{ is a leaf of }\tilde\F_1 \big\}<M_2,\ \text{and}$$
    $$\sup\big\{d_2(\tilde{e},\tilde{e}'): \tilde{e},\tilde{e}'\in (K\cap \widetilde{W}^2_1(\tilde{e})),\ \widetilde{W}^2_1\text{ is a leaf of }\tilde\F_2 \big\}<M_2,$$
    from Corollary \ref{distleaf1}.
    \item $M$: Let $M=10\max\{M_2,D\}$.
\end{itemize}

We could choose $\alpha<\varepsilon/10M$.

We denote $\calO:=\tilde{h}(\tilde{e}_0)$ and the linear leaves of $\tilde\F_i'',\ i=1,2$ passing through $\calO$ by $\widetilde{W}^1_0,\ \widetilde{W}^2_0$. There is then a natural coordinate of every point if we consider the projection of it onto these straight leaves. For a point $\tilde e\in\R^d$, we want to denote the distance between the projection of $\tilde e$ to $\widetilde W^1_0$ and $\calO$ along $\widetilde W^1_0$ as $\|\tilde e-\calO\|_1$ and the distance between the projection of $\tilde e$ to $\widetilde W^2_0$ and $\calO$ along $\widetilde W^2_0$ as $\|\tilde e-\calO\|_2$.

We claim that for any point $\tilde{e}_1\in\widetilde{W}_2^1(\tilde{e}_0)$ and any point $\tilde{e}_2\in\widetilde{W}_2^2(\tilde{e}_0)$ of the foliations $\tilde\F_i'$,
$$\text{either}\ \|\tilde{e}_1-\tilde{e}_2\|_1>0, \ \text{or}\ \|\tilde{e}_1-\tilde{e}_2\|_2>0.$$

\noindent\textbf{Case 1.} We start with the special case when $\|\tilde{e}_1-\calO\|_1$, $\|\tilde{e}_1-\calO\|_2$, $\|\tilde{e}_2-\calO\|_1$, and $\|\tilde{e}_2-\calO\|_2$ are all $\leq M_1+\varepsilon$.

There exist points $\tilde{e}_1'\in\widetilde{W}^1_1(\tilde{e}_0)$ and $\tilde{e}_2'\in\widetilde{W}^2_1(\tilde{e}_0)$ such that $\|\tilde{e}_1-\calO\|_1=\|\tilde{e}_1'-\calO\|_1$, $\|\tilde{e}_2-\calO\|_2=\|\tilde{e}_2'-\calO\|_2$. With the $\alpha$ we chose, $\widetilde{W}^2_2(\tilde{e}_1)$ stays in a cone of slope $\leq\frac{\varepsilon}{10M}$ around the leaf $\widetilde{W}^2_1(\tilde{e}_0)$ for at least distance of $M_2$, which is the upperbound of the distance along a leaf of $\tilde\F_i$ inside a set of diameter $D>10M_1$. We then have $\|\tilde{e}_2-\tilde{e}_2'\|_1\leq 2M_2\frac{\varepsilon}{10M}<<\varepsilon/2$. Also, $\|\tilde{e}_1-\tilde{e}_1'\|_1=0$.

Note that either $\|\tilde{e}_1'-\tilde{e}_2'\|_1\geq\varepsilon$, or $\|\tilde{e}_1'-\tilde{e}_2'\|_2\geq\varepsilon$ because of the local product structure and the fact that $\tilde{W}^2_1(\tilde{e}_0)$ and $\tilde{W}^1_1(\tilde{e}_0)$ (with the GPS) do not meet the second time besides at $\tilde{e}_0$. If both $\|\tilde e'_1-\tilde e_2'\|_1<\varepsilon$ and $\|\tilde e_1'-\tilde e_2'\|_2<\varepsilon$, this means $W_1^i(\tilde e_0)$ have entered a product neighborhood where they must intersect for another time, contradicting the global product structure.

If $\|\tilde{e}_1'-\tilde{e}_2'\|_1\geq\varepsilon$, then we have
$$\|\tilde{e}_1-\tilde{e}_2\|_1\geq \|\tilde{e}_1'-\tilde{e}_2'\|_1- \|\tilde{e}_1-\tilde{e}_1'\|_1-\|\tilde{e}_2-\tilde{e}_2'\|_1 \geq \varepsilon-2M_2\frac{\varepsilon}{10M}>0.$$
Similarly, if we have $\|\tilde{e}_1'-\tilde{e}_2'\|_2\geq\varepsilon$, then $\|\tilde{e}_1-\tilde{e}_2\|_2\geq\varepsilon$.

\noindent\textbf{Case 2.} When any of $\|\tilde{e}_1-\calO\|_1$, $\|\tilde{e}_1-\calO\|_2$, $\|\tilde{e}_2-\calO\|_1$, or $\|\tilde{e}_2-\calO\|_2$ is greater than $M_1+\varepsilon$, we can prove this claim by a double induction. See Figure \ref{figgpsopen}.

A brief idea of this is that we show leaves of $\tilde\F_i'$ follow along the leaves of $\tilde\F_i$ for distance $D$ with an error $<\varepsilon$. Then we swtich to other leaves of $\tilde\F_i$ to follow, in order to keep the error of our estimate small. At the same time note that the pair of leaves of $\tilde\F_i$ is far apart, so they never meet again by GPS, hence the leaves of $\tilde\F_i'$ would not meet neither.

\begin{figure}[H]\begin{tikzpicture}[scale=.9]
    \centering
  %horizontal
  \draw (-3,0) -- (10,0);
  \draw[gray,very thin] (-2.5,-1) -- (9.5,-1);
  \draw[gray,very thin] (-2.5,1) -- (10,1);
  \draw[gray,very thin] (-2,2) -- (10,2);
  \draw[gray,very thin] (-2,4.5) -- (10,4.5);
  %vertical
  \draw (-0.5,-1.5) -- (0,6);
  \draw[gray,very thin] (-1.5,-1.5) -- (-1,6);
  \draw[gray,very thin] (0.5,-1.5) -- (1,6);
  \draw[gray,very thin] (4.5,-1.5) -- (5,6);
  \draw[gray,very thin] (8.5,-1.5) -- (9,6);
  %unstable leaf of Anosov
  \draw plot[smooth, tension=0.7] coordinates{(-1.3,-1.1) (-0.8,0.7) (0.2,1.5) (0.88,5) (0.87,5.9)};
  %1st stable leaf of Anosov
  \draw plot[smooth, tension=0.7] coordinates{(-2,0.5) (-0.8,0.7) (0.5,0.9) (3,0.7) (5,0.95)};
  %2nd stable leaf of Anosov
  \draw plot[smooth, tension=0.7] coordinates{(4,0.9) (4.85,1) (7,1.6) (9.5,1.9)};
  %unstable of real
  \draw[gray] plot[smooth, tension=0.7] coordinates{(-1.2,-1) (-0.8,0.7) (0.3,1.45) (0.8,5.3)};
  %stable of real
  \draw[gray] plot[smooth, tension=0.7] coordinates{(-1.9,0.6) (-0.8,0.7) (0.5,0.8) (3,0.65) (4.63,0.95) (7,1.7) (9.2,2)};
  %labels
  \node[fill=white] at (-0.3,0) {\tiny $\calO$};
  \node at (-0.8,0.7)[circle,fill,inner sep=1pt]{};
  \node[] at (-0.95,0.87) {\tiny $\tilde{e}_0$};
  \node at (4.67,0.96)[circle,fill,inner sep=1pt]{};
  \node[] at (4.6,1.2) {\tiny $\tilde{e}^{(1)}_1$};
  \node at (8.73,1.95)[circle,fill,inner sep=1pt]{};
  \node[] at (8.8,2.2) {\tiny $\tilde{e}^{(2)}_1$};
  \node at (0.74,4.5)[circle,fill,inner sep=1pt]{};
  \node[] at (0.59,4.55) {\tiny $\tilde{e}^{(1)}_2$};
  %gray nodes
  %unstable
  \node at (0.69,4)[circle,fill,inner sep=1pt,gray]{};
  \node[] at (0.5,4)[gray] {\tiny $\tilde{e}_2$};
  %stable
  \node at (7,1.7)[circle,fill,inner sep=1pt,gray]{};
  \node[] at (6.9,1.9)[gray] {\tiny $\tilde{e}_1$};
  %mark length
  \draw[gray, <->] (-1.6,-1) -- (-1.53,0);
  \node at (-1.63,-0.5)[fill=white,rotate=86,font=\fontsize{4}{3.5}\selectfont]{$M_1+\varepsilon$};
  \draw[gray,<->] (-1.48,-1.15) -- (-0.49,-1.15);
  \node at (-1,-1.2)[fill=white,font=\fontsize{4}{3.5}\selectfont]{$M_1+\varepsilon$};
  %vertical D
  \draw[gray, <->] (-0.5,1) -- (-0.27,4.5);
  \node at (-0.41,2.7)[fill=white,font=\fontsize{4.5}{3.5}\selectfont,rotate=86]{$D$};
  %first D
  \draw[gray, <->] (0.58,-0.2) -- (4.59,-0.2);
  \node at (2.65,-0.2)[fill=white,font=\fontsize{4.5}{3.5}\selectfont]{$D$};
  %second D
  \draw[gray, <->] (4.59,-0.2) -- (8.6,-0.2);
  \node at (6.66,-0.2)[fill=white,font=\fontsize{4.5}{3.5}\selectfont]{$D$};
  % M_2
  \draw[gray] (4.69,0.85) -- (4.77,0.65);
  \draw[gray] (8.75,1.8) -- (8.75,1.58);
  \draw[gray,<->] plot[smooth, tension=0.7] coordinates{(4.72,0.75) (6.7,1.4) (8.75,1.7)};
  \node at (6.73,1.29)[fill=white,font=\fontsize{4.5}{3.5}\selectfont,rotate=15]{$\leq M_2$};
  %name of leaves
  \node[] at (-0.75,-1.5){\tiny $\widetilde{W}_0^2$};
  \node[] at (-2.5,0.2){\tiny $\widetilde{W}_0^1$};
  \node[] at (0.5,5.2)[gray]{\tiny $\widetilde{W}^2_2(\tilde{e}_0)$};
  \node[] at (8,2)[gray]{\tiny $\widetilde{W}^1_2(\tilde{e}_0)$};
  \node at (1.55,5.7) {\tiny $\widetilde{W}^2_1(\tilde{e}_0)$};
  \node at (1.6,1) {\tiny $\widetilde{W}^1_1(\tilde{e}_0)$};
  \node at (9.5,1.55) {\tiny $\widetilde{W}^1_1(\tilde{e}^{(1)}_1)$};
  %epsilon
  \draw[gray] (2.5,0.83) -- (2.7,0.825);
  \draw[gray] (2.5,0.55) -- (2.7,0.545);
  \draw[gray,dotted] (2.6,0.55) -- (2.6,0.83);
  \node at (2.6,0.7)[purple] {\tiny $\varepsilon$};
\end{tikzpicture}
\caption{Leaves} \label{figgpsopen}
\end{figure}

Here we see the necessity of the existence of the homeomorphism. If it is not injective, it may take a entire leaf to a bounded subset of the linear leaf. Then we would not have the leaves of $\tilde\F'_i$'s separating from each other.

The detailed induction is as the followings.

First for the base case, we let $\tilde e_1,\ \tilde e_2$ be points such that
$$M_1+\varepsilon<\|\tilde{e}_1-\calO\|_1\leq M_1+\varepsilon+D\ \ \text{and}\ \ \ M_1+\varepsilon<\|\tilde{e}_2-\calO\|_2\leq M_1+\varepsilon+D.$$
Note that we would not need to consider cases where $M_1+\varepsilon<\|\tilde{e}_1-\calO\|_1\leq M_1+\varepsilon+D$ and $\|\tilde{e}_2-\calO\|_2\leq M_1+\varepsilon$, or $\|\tilde{e}_1-\calO\|_1\leq M_1+\varepsilon$ and $M_1+\varepsilon<\|\tilde{e}_2-\calO\|_2\leq M_1+\varepsilon+D$, since they have been dealt with in Case 1.
For any such points $\tilde{e}_1\in \widetilde W_2^1(\tilde e_0)$ and $\tilde e_2\in \widetilde W_2^2(\tilde e_0)$, we again have $\tilde{e}_1'\in\widetilde{W}^1_1(\tilde{e}_0)$ and $\tilde{e}_2'\in\widetilde{W}^2_1(\tilde{e}_0)$, such that $\|\tilde{e}_1-\calO\|_1=\|\tilde{e}_1'-\calO\|_1$ and $\|\tilde{e}_2-\calO\|_2=\|\tilde{e}_2'-\calO\|_2$.

We have the estimate $\|\tilde{e}_1-\tilde{e}_1'\|_2<4M_2\alpha<4M_2\frac{\varepsilon}{10M}<\varepsilon$, and similarly $\|\tilde{e}_2-\tilde{e}_2'\|_1<\varepsilon$, while $\|\tilde e_1-\tilde e_1'\|_1=\|\tilde e_2-\tilde e_2'\|_2=0$. We have $\|\tilde{e}_1'-\calO\|_2\leq M_1$ because again $\tilde{h}$ has a bounded distance $M_1$ from the identity, and as a result any point on $\widetilde{W}^1_1(\tilde{e}_0)$ has a bounded distance $M_1$ from the linear leaf $\widetilde W_0^1$. In this case we always have $\|\tilde{e}_1-\calO\|_2<M_1+\varepsilon$ and $\|\tilde{e}_2-\calO\|_1<M_1+\varepsilon$ because for instance,
$$\|\tilde{e}_1-\calO\|_2\leq \|\tilde{e}_1-\tilde{e}_1'\|_2+\|\tilde{e}_1'-\calO\|_2 <M_1+\varepsilon.$$
Then
\begin{align*} \|\tilde{e}_1-\tilde{e}_2\|_1 \geq\|\tilde{e}_1-\calO\|_1-\|\tilde{e}_2-\calO\|_1>(M_1+\varepsilon)-(M_1+\varepsilon)>0.
\end{align*}

Now we denote by $\tilde{e}^{(k)}_1$ the point on $\widetilde{W}^1_2(\tilde{e}_0)$ such that $\|\tilde{e}_1^{(k)}-\calO\|_1=M_1+kD<(\frac{1}{10}+k)D$; $\tilde{e}^{(k)}_2$ the point on $\widetilde{W}^2_2(\tilde{e}_0)$ such that $\|\tilde{e}_2^{(k)}-\calO\|_2=M_1+kD<(\frac{1}{10}+k)D$, $k\geq 1$. We then take leaves $\widetilde{W}^1_1(\tilde{e}_1^{(k)})$ and $\widetilde{W}^2_1(\tilde{e}_2^{(k)})$ for each $k\geq 1$. These are the ``new leaves'' of $\tilde\F_i$ we switch to track.

The induction hypothesis: For $i\geq 1,\ j\geq 1$, and such selected $\tilde e^{(i)}_1,\ \tilde e^{(j)}_2$,
we have
$$\|\tilde e_1^{(i)}-\calO\|_2< (2i-1)(M_1+\varepsilon),\ \text{and}\ \|\tilde e_2^{(j)}-\calO\|_1< (2j-1)(M_1+\varepsilon).$$
Note that this is implicit in the base case.

Now consider the case where $i\geq 1$ and $j\geq 1$, and the points $\tilde e_1,\ \tilde e_2$ such that
$$M_1+\varepsilon+iD<\|\tilde e_1-\calO\|_1\leq M_1+\varepsilon +(i+1)D,\ \text{and}$$
$$M_1+\varepsilon+jD<\|\tilde e_2-\calO\|_2\leq M_1+\varepsilon +(j+1)D.$$

For any $\tilde{e}_1\in \widetilde W^1_2(\tilde e_0)$ and $\tilde{e}_2\in \widetilde W^2_2(\tilde e_0)$,
we could also find $\tilde{e}_1'\in\widetilde{W}^1_1(\tilde{e}_1^{(i)})$ and $\tilde{e}_2'\in\widetilde{W}^2_1(\tilde{e}_2^{(j)})$ such that $\|\tilde{e}_1-\calO\|_1=\|\tilde{e}_1'-\calO\|_1$ and $\|\tilde{e}_2-\calO\|_2=\|\tilde{e}_2'-\calO\|_2$.

Again, $\|\tilde{e}_1-\tilde{e}_1'\|_2\leq 2M_2\alpha<2M_2\frac{\varepsilon}{10M}<\varepsilon$ and $\|\tilde{e}_2-\tilde{e}_2'\|_1 <\varepsilon$. Here we have
\begin{align*}
\|\tilde{e}_1-\calO\|_2 &\leq\|\tilde{e}_1-\tilde{e}_1'\|_2+\|\tilde{e}_1'-\tilde{e}_1^{(i)}\|_2+\|\tilde{e}_1^{(i)}-\calO\|_2\\
&\leq \varepsilon+2M_1+(2i-1)(M_1+\varepsilon)\\
&< (2i+1)(M_1+\varepsilon),
\end{align*}
where $\|\tilde e_1'-\tilde e_1^{(i)}\|_2\leq 2M_1$ because every point on the leaf $\widetilde W_1^1(\tilde e_1^{(i)})$ is of bounded distance $<M_1$ from a linear leaf passing through $h(\tilde e_1^{(1)})$ and then $\tilde e'_1$ can only be in a cylinder of radius $M_1$ that contains this linear leaf.
Similarly, we have $\|\tilde{e}_2-\calO\|_1< (2j+1)(M_1+\varepsilon)$. 

If $i\geq j$, then
\begin{align*} \|\tilde{e}_1-\tilde{e}_2\|_1 &\geq\|\tilde{e}_1-\calO\|_1-\|\tilde{e}_2-\calO\|_1 \\
& >\|\tilde{e}_1^{(i)}-\calO\|_1-\|\tilde{e}_2-\calO\|_1\\
& \geq iD-(2j+1)(M_1+\varepsilon)\\
&\geq (10i-(2j+1))(M_1+\varepsilon)>0.
\end{align*}
If $i\leq j$, then $\|\tilde{e}_1-\tilde{e}_2\|_2>0$.
\end{proof}

\begin{corollary}\label{gpsopen} For our fibrewise Anosov diffeomorphism $F:E\to E$, suppose at $b_0\in B$ we have the global product structure. Then there exists a $\delta$ (uniform for all $b_0\in B$) such that if $d_B(b_0,b)<\delta$, then we also have the global product structure at $b$.
\end{corollary}
\begin{proof} We apply the above proposition, where $\F_i,\ i=1,2$ is the pair of the stable and unsatble foliations of $F_{b_0}$ and $\F_i',\ i=1,2$ is the stable and unstable foliations of $F_{b}$. We get an $\alpha$.

Because of the smoothness of $F$ and thus of its induced foliations and the compactness of the bundle, we are able to pick a $\delta>0$ such that if $d(b,b_0)<\delta$ with $b,b_0$ in the same trivialized neighborhood, then at any point $e\in p^{-1}b_0$ and $e'\in p^{-1}b$ such that $\psi(e)=\psi(e')$, we have $\max \angle_e(p_i(e),d\psi(p_i'(e)))<\alpha$, where $\{p_i\}_{i=1}^d$ and $\{p_i'\}_{i=1}^d$ are bases of the tangent space of $p^{-1}b_0$ and $p^{-1}b$ respectively. Also, $p_i,\ p_i'$ are bases elements of the stable bundle if $i=1,\ldots, l$ and are bases elements of the unstable bundle if $i=l+1,\ldots,d$.
\end{proof}

\begin{corollary}\label{gpsclosed} Suppose $\{b_n\}$ is a sequence of points in the base $B$ that converges to $b_0$, where at each $b_n$ the stable and unstable foliations of $F$ have the global product structure. Then at $b_0$ we also have the global product structure.
\end{corollary}
\begin{proof} This also follows from Corollary \ref{gpsopen}, because for that fixed $\delta$, we could find an $N$ such that $d(b_0,b_N)<\delta$.
\end{proof}

We prove Theorem \ref{main} by showing that the $h$ itself is a homeomorphism.

\begin{proof}[Proof of Theorem \ref{main}] 
The set
$$\{b\in B\ |\ \text{The foliations induced by $F_b$ has the global product structure}\}$$
is nonempty by Lemma \ref{recgps}, is open by Corollary \ref{gpsopen}, and is closed by Corollary \ref{gpsclosed}. Hence  this set is the full set of $B$. By Lemma \ref{hinj}, $\tilde{h}_b$ is injective for all $b\in B$.

We know $h$ is surjective from Proposition \ref{semii}. If we take $e,e'$ such that $h(e)=h(e')$, then $e$ and $e'$ must be in the same fibre. Let $b:=p(e)=p(e')$. Consider the lifts to $\R^d_b$. Then there is a deck transformation $j\in\Z^d$, such that $\tilde{h}_b(\tilde{e})=\tilde{h}_b(\tilde{e}')+j=\tilde{h}_b(\tilde{e}'+j)$, so $\tilde{e}=\tilde{e}'+j$. Thus $e=e'$. Therefore $h$ is also injective.
\end{proof}

\

\noindent\textbf{Acknowledgements.} I would like to thank my advisor Andrey Gogolev for suggesting this problem, many discussions, and carefully reading my writing.

\

\nocite{*}
\bibliographystyle{alpha}
\bibliography{bibfile}

\

\address{\parbox{\linewidth}{
  Department of Mathematics - The Ohio State University,\\
  100 Math Tower, 231 West 18th Avenue, Columbus, OH 43210-1174, USA}}
\email{zhang.8939@osu.edu}

%\printindex

%\clearpage

\end{document}